\documentclass[preprint,3p,12pt]{elsarticle}
\usepackage[english]{babel}
\usepackage{amsmath}
\usepackage{amsthm}
\usepackage{amstext}
\usepackage{amssymb}
\usepackage{mathtools}

\def\rF{\mathbb{F}}
\def\R{\mathbb{R}}

\def\argmin{\mathop{\rm arg\, min}}

\def\cE{\mathbb{E}}

\def\B{{\mathcal B}}

\def\P{{\mathcal P}}
\def\Q{{\mathcal Q}}

\def\sPr{{\mathsf{Pr}}}

\def\sX{{\mathsf X}}
\def\sY{{\mathsf Y}}
\def\sZ{{\mathsf Z}}
\def\sH{{\mathsf H}}
\def\sA{{\mathsf A}}
\def\sE{{\mathsf E}}

\newtheorem{theorem}{Theorem}[section]

\newtheorem{lemma}{Lemma}[section]

\theoremstyle{definition}
\newtheorem{definition}{Definition}[section]
\newtheorem{remark}{Remark}[section]
\newtheorem{example}{Example}[section]
\newtheorem{assumption}{Assumption}[section]

\allowdisplaybreaks

\begin{document}

\begin{frontmatter}

\title{Near Optimality of Quantized Policies in Stochastic Control Under Weak Continuity Conditions\tnoteref{label0}}
\tnotetext[label0]{This research was supported in part by the Natural Sciences and Engineering Research Council (NSERC) of Canada. Parts of this work were presented at 53rd IEEE Conference on Decision and Control (CDC), Los Angeles, CA, 2014.}

\author{Naci Saldi\corref{cor}}
\ead{nsaldi@mast.queensu.ca}
\cortext[cor]{Corresponding author}

\author{Serdar Y\"uksel}
\ead{yuksel@mast.queensu.ca}

\author[label1]{Tam\'{a}s Linder}
\address[label1]{Department of Mathematics and Statistics, Queen's University, Kingston, ON, Canada.}
\ead{linder@mast.queensu.ca}

\begin{abstract}
This paper studies the approximation of optimal control policies by quantized (discretized) policies for a very general class of Markov decision processes (MDPs). The problem is motivated by applications in networked control systems, computational methods for MDPs, and learning algorithms for MDPs. We consider the finite-action approximation of stationary policies for a discrete-time Markov decision process with discounted and average costs under a weak continuity assumption on the transition probability, which is a significant relaxation of conditions required in earlier literature. The discretization is constructive, and quantized policies are shown to approximate optimal deterministic stationary policies with arbitrary precision. The results are applied to the fully observed reduction of a partially observed Markov decision process, where weak continuity is a much more reasonable assumption than more stringent conditions such as strong continuity or continuity in total variation.
\end{abstract}

\begin{keyword}
Stochastic control \sep quantization \sep approximation \sep partially observed Markov decision processes
\end{keyword}

\end{frontmatter}

\section{Introduction}
\label{sec0}

In this paper, we study the finite-action approximation of optimal control policies for a discrete time Markov decision process (MDP) with Borel state and action spaces, under discounted and average cost criteria. Various stochastic control problems may benefit from such an investigation.

The optimal information transmission problem in networked control systems is one such example. In many applications to networked control, the perfect transmission of the control actions to an actuator is infeasible when there is a communication channel of finite capacity between a controller and an actuator. Hence, the actions of the controller must be quantized to facilitate reliable transmission to an actuator. Although, the problem of optimal information transmission from a plant/sensor to a controller has been studied extensively \cite{YuBa13}, much less is known about the problem of transmitting actions from a controller to an actuator. Such transmission schemes usually require a simple encoding/decoding rule since an actuator does not have the computational/intelligence capability of a controller to use complex algorithms. Therefore, time-invariant uniform quantization is a practically useful encoding rule for controller-actuator communication.

The investigation of the finite-action approximation problem is also useful in computing near optimal policies and learning algorithms for MDPs. In a recent work \cite{SaYuLi14}, we consider the development of \emph{finite-state} approximations for obtaining near optimal policies. However, to establish constructive control schemes, one needs to quantize the action space as well. Thus, results on approximate optimality of finite-action models paves the way for practical computation algorithms which are commonly used for finite-state/action MDPs. One other application regarding approximation problems is on learning a controlled Markov chain using simulations. If one can ensure that learning a control model with only finitely many control actions is sufficient for approximate optimality, then it is easier to develop efficient learning algorithms which allow for the approximate computation of finitely many transition probabilities. In particular, results developed  in the learning and information theory literature for conditional kernel estimations \cite{LaKo07} (with control-free models) can be applied to transition probability estimation for MDPs.

Motivated as above, in this paper we investigate the following approximation problem: For uncountable Borel state and action spaces, under what conditions can the optimal performance (achieved by some optimal stationary policy) be arbitrarily well approximated if the controller action set is restricted to be finite? We show that quantized stationary policies obtained by uniform quantization of the action space can approximate optimal policies with arbitrary precision for an MDP with an unbounded one-stage cost function, under a weak continuity assumption on the transition probability.

Various approximation results, which are somewhat related to our work, have been established for MDPs with Borel state and action spaces in the literature along the theme of computation of near optimal policies. With the exception of \cite{Lan81}, these works assume in general restrictive continuity conditions on the transition probability. In \cite{Lan81}, the authors considered an approximation problem in which all the components of the original model are \emph{allowed} to vary in the reduced model (varying only the action space corresponds to the setup considered in this paper). Under weak continuity of the transition probability, \cite{Lan81} established the convergence of the reduced models to the original model for the discounted cost when the one-stage cost function is bounded. In this paper we allow the one-stage cost function to be unbounded. In addition, we also study the approximation problem for the challenging average cost case. Hence, our results can be applied to a wider range of stochastic systems. However, analogous with \cite{Lan81}, the price we pay for imposing weaker assumptions is that we do not obtain explicit performance bounds in terms of the rate of the quantizer used in the approximations.

In \cite{SaLiYu13-2} we solved a variant of this problem for the discounted cost under the folowing assumptions: (i) the action space is compact, (ii) the transition probability is setwise continuous in the action variable, and (iii) the one stage cost function is bounded and continuous in the action variable. The average cost was also considered under some additional restrictions on the ergodicity properties of Markov chains induced by deterministic stationary policies. In this paper we consider a similar problem for systems where the transition probability is weakly continuous in the state-action variables. An important motivation for considering these conditions comes from the fact that for the fully observed reduction of a partially observed MDP (POMDP), the setwise continuity of the transition probability in the action variable is a prohibitive condition even for simple systems such as the one described in Example~\ref{exm1} in the next section.

\emph{Organization:} In Section~\ref{sec1} we give definitions and the problem formulation. The main result for discounted cost is stated and proved in Section~\ref{sec2}. In Section~\ref{sec2sub2} an analogous approximation result is obtained for the average cost criterion. In Section~\ref{sec3} the results for the discounted cost are applied to the fully observed reduction of POMDPs via appealing to results by Feinberg \emph{et al.\ }\cite{FeKaZg14}. Section~\ref{sec4} concludes the paper.

\section{Markov Decision Processes}
\label{sec1}

For a metric space $\sE$, the Borel $\sigma$-algebra is denoted by $\B(\sE)$. Given any Borel measurable function $w: \sE \rightarrow [1,\infty)$ and any real valued Borel measurable function $u$ on $\sE$, we define $w$-norm of $u$ as
\begin{align}
\|u\|_{w} \coloneqq \sup_{e\in\sE} \frac{|u(e)|}{w(e)}. \nonumber
\end{align}
Let $C_w(\sE)$ and $B_w(\sE)$ denote the space of all real valued continuous and Borel measurable functions on $\sE$ with finite $w$-norm, respectively \cite{HeLa99}. Let $\P(\sE)$ denote the set of all probability measures on $\sE$.
A sequence $\{\mu_n\}$ of measures on $\sE$ is said to converge weakly (resp., setwise) \cite{HeLa03} to a measure $\mu$ if $\int_{\sE} g(e) \mu_n(de)\rightarrow\int_{\sE} g(e) \mu(de)$ for all bounded and continuous real functions $g$ on $\sE$ (resp., for all bounded and Borel measurable real functions $g$ on $\sE$). Unless otherwise specified, the term "measurable" will refer to Borel measurability.

We consider a discrete-time Markov decision process (MDP) with \emph{state space} $\sZ$ and \emph{action space} $\sA$, where
$\sZ$ and $\sA$ are Borel spaces. In this paper, we assume that the \emph{set of admissible actions} for any $z \in \sZ$ is $\sA$. Let the \emph{stochastic kernel} $\eta(\,\cdot\,|z,a)$ be the \emph{transition probability} of the next state given that the previous state-action pair is $(z,a)$ \cite{HeLa96}. The \emph{one-stage cost} function $c$ is a measurable function from $\sZ \times \sA$ to $[0,\infty)$. The probability measure $\xi$ over $\sZ$ denotes the initial distribution of the state process.

Define the history spaces $\sH_0 = \sZ$ and
$\sH_{t}=(\sZ\times\sA)^{t}\times\sZ$, $t=1,2,\ldots$ endowed with their
product Borel $\sigma$-algebras generated by $\B(\sZ)$ and $\B(\sA)$. A
\emph{policy} is a sequence $\varphi=\{\varphi_{t}\}$ of stochastic kernels
on $\sA$ given $\sH_{t}$. The set of all policies is denoted by $\Phi$. A \emph{deterministic} policy is a
sequence $\varphi=\{\varphi_{t}\}$ of stochastic kernels on $\sA$ given $\sH_{t}$ which are realized by a
sequence of measurable functions $\{f_{t}\}$ from $\sH_{t}$ to $\sA$, i.e.,
$\varphi_{t}(\,\cdot\,|h_{t})=\delta_{f_{t}(h_{t})}(\,\cdot\,)$, where $\delta_z$ is the point mass at $z$.
Let $\rF$ denote the set of all measurable functions $f$ from $\sZ$ to $\sA$.
A \emph{deterministic stationary} policy is a constant sequence of stochastic kernels $\varphi=\{\varphi_{t}\}$ on $\sA$ given
$\sZ$ such that $\varphi_{t}(\,\cdot\,|z)=\delta_{f(z)}(\,\cdot\,)$ for all $t$ for some
$f\in\rF$. The set of deterministic stationary policies is identified with the set $\rF$.

According to the Ionescu Tulcea theorem \cite{HeLa96}, an initial distribution $\xi$ on $\sZ$ and a policy $\varphi$ define a unique probability measure $P_{\xi}^{\varphi}$ on $\sH_{\infty}=(\sZ\times\sA)^{\infty}$. The expectation with respect to $P_{\xi}^{\varphi}$ is denoted by $\cE_{\xi}^{\varphi}$.
If $\xi=\delta_z$, we write $P_{z}^{\varphi}$ and $\cE_{z}^{\varphi}$ instead of $P_{\delta_z}^{\varphi}$ and $\cE_{\delta_z}^{\varphi}$. The cost functions to be minimized in this paper are the discounted cost with a discount factor $\beta \in (0,1)$ and the average cost, respectively:
\begin{align}
J(\varphi,\xi) &= \cE_{\xi}^{\varphi}\biggl[\sum_{t=0}^{\infty}\beta^{t}c(z_{t},a_{t})\biggr], \nonumber \\
V(\varphi,\xi) &= \limsup_{T \rightarrow \infty} \frac{1}{T} \cE_{\xi}^{\varphi}\biggl[\sum_{t=0}^{T-1}c(z_{t},a_{t})\biggr]. \nonumber
\end{align}
A policy $\varphi^{*}$ is called optimal if $J(\varphi^{*},z) = \inf_{\varphi \in \Phi} J(\varphi,z)$ (or $V(\varphi^{*},z) = \inf_{\varphi \in \Phi} V(\varphi,z)$ for the average cost) for all $z \in \sZ$. It is well known that the set $\rF$ of deterministic stationary policies contains an optimal policy for a large class of infinite horizon discounted cost (see, e.g., \cite{HeLa96,FeKaZa12}) and average cost optimal control problems (see, e.g., \cite{Bor02,FeKaZa12}). In this case we say that $\rF$ is an optimal class.

\subsection{Problem Formulation}
\label{sec1sub1}

To give a precise definition of the problem we study in this paper, we first give the definition of a quantizer from the state to the action space.

\begin{definition}
A measurable function $q: \sZ \rightarrow\sA$ is called a \emph{quantizer} from
$\sZ$ to $\sA$ if the range of $q$, i.e., $q(\sZ)=\{q(z)\in\sA:z\in\sZ\}$, is
finite.
\end{definition}

The elements of $q(\sZ)$ (the possible values of $q$) are called
the \emph{levels} of $q$. The rate $R = \log_2|q(\sZ)|$ of a quantizer $q$ (approximately)
represents the number of bits needed to losslessly encode the output levels of $q$ using binary
codewords of equal length. Let $\Q$ denote the set
of all quantizers from $\sZ$ to $\sA$. A \emph{deterministic stationary quantizer} policy
is a constant sequence $\varphi=\{\varphi_{t}\}$ of stochastic kernels on $\sA$ given
$\sZ$ such that $\varphi_{t}(\,\cdot\,|z)=\delta_{q(z)}(\,\cdot\,)$ for all $t$ for
some $q\in\Q$. For any finite set $\Lambda \subset \sA$, let $\Q(\Lambda)$ denote the set of all elements in $\Q$ having range
$\Lambda$. Analogous with $\rF$, the set of all deterministic stationary quantizer policies induced by $\Q(\Lambda)$ will be identified with the set $\Q(\Lambda)$.

Our main objective is to find conditions on the components of the MDP under which there exists a sequence of finite subsets $\{\Lambda_n\}_{n\geq1}$ of $\sA$ for which the following holds:

\textbf{(P)} For any initial point $z$ we have $\lim_{n\rightarrow\infty} \min_{q \in \Q(\Lambda_n)} J(q,z)=\min_{f \in \rF} J(f,z)$
(or $\lim_{n\rightarrow\infty} \min_{q \in \Q(\Lambda_n)} V(q,z)=\min_{f\in \rF} V(f,z)$ for the average cost), provided that the set $\rF$ of deterministic stationary policies is an optimal class for the MDP.

Letting MDP$_n$ be defined as the Markov decision process having the components $\bigl\{ \sZ,\Lambda_n,\eta,c \bigr\}$, \textbf{(P)} is equivalent to stating that optimal cost of MDP$_n$ converges to the optimal cost of the original MDP.

\subsection{Setwise Continuity versus Weak Continuity}\label{compare}

Requiring that the transition probability of the system be weakly continuous in state-action variables is indeed a fairly mild assumption compared to the setwise continuity in the action variable. Indeed, the latter condition is even prohibitive for certain stochastic systems. In this section we consider two examples, the fully observed reduction of a POMDP and the additive noise model, in order to more explicitly highlight this. We refer the reader to \cite[Chapter 4]{Her89} and Section~\ref{sec3} of this paper for the basics of POMDPs.

\begin{example}\label{exm1}

Consider the system dynamics
\begin{align}
x_{t+1} &= x_t + a_t, \nonumber\\
y_t &= x_t + v_t, \nonumber
\end{align}
where $x_t \in \sX$, $y_t \in \sY$ and $a_t \in \sA$, and where $\sX$, $\sY$ and $\sA$ are the state, observation and action spaces, respectively. Here, we assume that $\sX = \sY = \sA = \R_{+}$ and the `noise' $\{v_t\}$ is a sequence of i.i.d. random variables uniformly distributed on $[0,1]$. It is easy to see that the transition probability is weakly continuous (with respect to state-action variables) and the observation channel that gives the conditional distribution of the current observation given the current state is continuous in total variation (with respect to state variable) for this POMDP. Hence, by \cite[Theorem 3.7]{FeKaZg14} the transition probability $\eta$ of the fully observed reduction of the POMDP is weakly continuous in the state-action variables. However, the same conclusion cannot be drawn for the setwise continuity of $\eta$ with respect to the action variable as shown below.

Let $z$ denote the generic state variable for the fully observed reduced MDP, where the state variables are elements of $\P(\sX)$ which is equipped with the Borel $\sigma$-algebra generated by the topology of weak convergence. If we define the function $F(z,a,y) \coloneqq \Pr\{x_{t+1} \in \,\cdot\, | z_t = z, a_t = a, y_{t+1} = y\}$ from $\P(\sX)\times\sA\times\sY$ to $\P(\sX)$ and the stochastic kernel $H(\,\cdot\, | z,a) \coloneqq \Pr\{y_{t+1} \in \,\cdot\, | z_t = z, a_t = a\}$ on $\sY$ given $\P(\sX)\times\sA$, then $\eta$ can be written as
\begin{align}
\eta(\,\cdot\,|z,a) = \int_{\sY} 1_{\{F(z,a,y) \in \,\cdot\,\}} H(dy|z,a), \nonumber
\end{align}
where $1_D$ denotes the indicator function of an event $D$ and $z_t$ denotes the posterior distribution of the state $x_t$ given the past observations, i.e.,
\begin{align}
z_t(\,\cdot\,) \coloneqq \Pr \{x_t \in \,\cdot\,| y_0,\ldots,y_t,a_0,\ldots,a_{t-1}\}.\nonumber
\end{align}
Let us set $z = \delta_0$ (point mass at $0 \in \sX$), $\{a_k\} = \{\frac{1}{k}\}$, and $a = 0$. Hence, $a_k \rightarrow a$. We show that $\eta(\,\cdot\,|z,a_k) \nrightarrow \eta(\,\cdot\,|z,a)$ setwise.

Observe that for all $k$ and $y \in \sY$, we have $F(z,a_k,y) = \delta_{\frac{1}{k}}$ and $F(z,a,y) = \delta_0$. Define the open set $O$ with respect to the weak topology in $\P(\sX)$ as
\begin{align}
O \coloneqq \bigl\{z \in \P(\sX): \bigl|\int_{\sX} g(x) \delta_1(dx) - \int_{\sX} g(x) z(dx)\bigr| < 1\bigr\}, \nonumber
\end{align}
where $g$ is the symmetric triangular function on $[-1,1]$ with $g(0) = 1$ and $g(-1) = g(1) = 0$. Observe that we have $F(z,a_k,y) \in O$ for all $k$ and $y$, but $F(z,a,y) \not\in O$ for all $y$. Hence,
\begin{align}
\eta(O|z,a_k) &\coloneqq \int_{\sY} 1_{\{F(z,a_k,y) \in O\}}  H(dy|z,a_k) = 1, \nonumber \\
\intertext{but}
\eta(O|z,a) &\coloneqq \int_{\sY} 1_{\{F(z,a,y) \in O\}}  H(dy|z,a) = 0 \nonumber
\end{align}
implying that $\eta(\,\cdot\,|z,a_k)$ does not converge to $\eta(\,\cdot\,|z,a)$ setwise. Hence, $\eta$ does not satisfy the setwise continuity assumption.
\end{example}

\begin{example}\label{exm2}
In this example we consider an additive-noise system given by
\begin{align}
z_{t+1}=F(z_{t},a_{t})+v_{t}, \text{ } t=0,1,2,\ldots \nonumber
\end{align}
where $\sZ= \R^n$ and  $\sA = \R^m$ for some $n,m\geq1$. The noise process $\{v_{t}\}$ is a sequence of independent and identically distributed (i.i.d.) random vectors. In such a system, the continuity of $F$ in $(z,a)$ (which holds for most practical systems) is sufficient to imply the weak continuity of the transition probability, and no assumptions are needed on the noise process (not even the existence of a density is required). Hence, weak continuity does not restrict the noise model and is satisfied by almost all systems in the applications, whereas other conditions, such as strong continuity or continuity in total variation distance, hold only if the noise is well behaved in addition to the continuity of $F$. For example, for setwise continuity, it is usually required that $F$ is continuous in $a$ for every $z$, the noise admits a density, and this density is continuous \cite[Example C.8]{HeLa96}.
\end{example}

\section{Near Optimality of Quantized Policies with Discounted Cost}
\label{sec2}

In this section we consider the problem \textbf{(P)} for the discounted cost. The following assumptions will be imposed for both the discounted cost and the average cost. We note that these assumptions are often used in the literature for studying discounted Markov decision processes with unbounded one-stage cost and weakly continuous transition probability.

\begin{assumption}
\label{as1}
\item [(a)] The one stage cost function $c$ is nonnegative and continuous.
\item [(b)] The stochastic kernel $\eta(\,\cdot\,|z,a)$ is weakly continuous in $(z,a) \in \sZ \times \sA$, i.e., if $(z_k,a_k) \rightarrow (z,a)$, then $\eta(\,\cdot\,|z_k,a_k) \rightarrow \eta(\,\cdot\,|z,a)$ weakly.
\item [(c)] $\sA$ is compact.
\item [(d)] There exist nonnegative real numbers $M$ and $\alpha \in [1,\frac{1}{\beta})$, and a continuous weight function $w:\sZ\rightarrow[1,\infty)$ such that for each $z \in \sZ$, we have
\begin{align}
\sup_{a \in \sA} c(z,a) &\leq M w(z), \label{eq1} \\
\sup_{a \in \sA} \int_{\sZ} w(y) \eta(dy|z,a) &\leq \alpha w(z), \label{eq2}
\end{align}
and $\int_{\sZ} w(y) \eta(dy|z,a)$ is continuous in $(z,a)$.
\end{assumption}

Let $d_{\sA}$ denote the metric on $\sA$. Since $\sA$ is compact and thus totally bounded, one can find a sequence of finite sets $\Lambda_n = \{a_{n,1},\ldots,a_{n,k_n}\} \subset \sA$ such that for all $n$,
\begin{align}
\min_{i\in\{1,\ldots,k_n\}} d_{\sA}(a,a_{n,i}) < 1/n \text{ for all } a \in \sA. \label{eq12}
\end{align}
In other words, $\Lambda_n$ is a $1/n$-net in $\sA$. In the rest of this paper, we assume that the sequence $\{\Lambda_n\}_{n\geq1}$ is fixed. To ease the notation in the sequel, let us define the mapping $\Upsilon_n : \rF \rightarrow \Q(\Lambda_n)$ by
\begin{align}
\Upsilon_n(f)(z) = \argmin_{a\in\Lambda_n} d_{\sA}(f(z),a) \label{neq6}.
\end{align}
Hence, for all $f\in\rF$, we have
\begin{align}
\sup_{z\in\sZ} d_{\sA}\bigl(\Upsilon_n(f)(z),f(z)\bigr) < 1/n \label{neq3}.
\end{align}

Define the operator $T$ on the set of real-valued measurable functions on $\sZ$ by
\begin{align}
T u(z) \coloneqq \min_{a \in \sA} \biggl[ c(z,a) + \beta \int_{\sZ} u(y) \eta(dy|z,a) \biggr] \label{eq9}.
\end{align}
In the literature $T$ is called the \emph{Bellman optimality operator}.

\begin{lemma}\label{nlemma1}
For any $u \in C_w(\sZ)$ the function
$l_u(z,a) \coloneqq \int_{\sZ} u(y) \eta(dy|z,a)$ is continuous in $(z,a)$.
\end{lemma}

\begin{proof}
For any nonnegative continuous function $g$ on $\sZ$, the function $l_g(z,a) = \int_{\sZ} g(y) \eta(dy|z,a)$ is lower semi-continuous in $(z,a)$, if $\eta$ is weakly continuous (see, e.g., \cite[Proposition E.2]{HeLa96}). Define the nonnegative continuous function $g$ by letting $g = b w + u$, where $b = \|u\|_w$. Then $l_g$ is lower semi-continuous. Since $l_u = l_g - b l_w$ and $l_w$ is continuous by Assumption~\ref{as1}-(d), $l_u$ is lower semi-continuous. Analogously, define the nonnegative continuous function $v$ by letting $v = -u + b w$. Then $l_v$ is lower semi-continuous. Since $l_u = b l_w - l_v$ and $l_w$ is continuous by Assumption~\ref{as1}-(d), $l_u$ is also upper semi-continuous. Therefore, $l_u$ is continuous.
\end{proof}

\begin{lemma}\label{nlemma2}
Let $\sY$ be any of the compact sets $\sA$ or $\Lambda_n$. Define the operator $T_{\sY}$ on $B_w(\sZ)$ by letting
\begin{align}
T_{\sY}u(z) \coloneqq \min_{a\in\sY} \biggl[ c(z,a) + \beta \int_{\sZ} u(y) \eta(dy|z,a) \biggr]. \nonumber
\end{align}
Then $T_{\sY}$ maps $C_w(\sZ)$ into itself. Moreover, $C_w(\sZ)$ is closed with respect to the $w$-norm.
\end{lemma}

\begin{proof}
Note that $T_{\sY}u(z) = \min_{a\in\sY} \bigl( c(z,a) + \beta l_u(z,a) \bigr)$. The function $l_u$ is continuous by Lemma~\ref{nlemma1}, and therefore, $T_{\sY}u$ is also continuous by \cite[Proposition 7.32]{BeSh78}. Since $T_{\sY}$ maps $B_w(\sZ)$ into itself, $T_{\sY} u \in C_w(\sZ)$.

For the second statement, let $u_n$ converge to $u$ in $w$-norm in $C_w(\sZ)$. It is sufficient to prove that $u$ is continuous. Let $z_k \rightarrow z_0$. Since $B \coloneqq \{z_0,z_1,z_2,\ldots\}$ is compact, $w$ is bounded on $B$. Therefore, $u_n \rightarrow u$ uniformly on $B$ which implies that $\lim_{k\rightarrow\infty} u(z_k) = u(z_0)$. This completes the proof.
\end{proof}

Lemma~\ref{nlemma2} implies that $T$ maps $C_w(\sZ)$ into itself. It can also be proved that $T$ is a contraction operator with modulus $\sigma \coloneqq \beta\alpha$ (see \cite[Lemma 8.5.5]{HeLa99}); that is,
\begin{align}
\|T u -  T v\|_w \leq \sigma \|u - v\|_w \text{ for all } u,v \in C_w(\sZ). \nonumber
\end{align}
Define the discounted value function $J^{*}$ by
\begin{align}
J^{*}(z) \coloneqq \inf_{\varphi \in \Phi} J(\varphi,z). \nonumber
\end{align}
The following theorem is a known result in the theory of Markov decision processes (see e.g., \cite[Section 8.5, p. 65]{HeLa99}).

\begin{theorem}
\label{thm1}
Suppose Assumption~\ref{as1} holds. Then, the value function $J^{*}$ is the unique fixed point in $C_w(\sZ)$ of the contraction operator $T$, i.e.,
\begin{align}
J^{*} = T J^{*}. \label{eq10}
\end{align}
Furthermore, a deterministic stationary policy $f^{*}$ is optimal if and only if
\begin{align}
J^{*}(z) = c(z,f^{*}(z)) + \beta \int_{\sZ} J^{*}(y) \eta(dy|z,f^{*}(z)). \label{eq11}
\end{align}
Finally, there exists a deterministic stationary policy $f^{*}$ which is optimal, so it satisfies~(\ref{eq11}).
\end{theorem}

Define, for all $n\geq1$, the operator $T_n$ (which will be used to approximate $T$) by
\begin{align}
T_n u(z) \coloneqq \min_{a\in\Lambda_n} \biggl[ c(z,a) + \beta \int_{\sZ} u(y) \eta(dy|z,a) \biggr] \label{eq13}.
\end{align}
Note that $T_n$ is the Bellman optimality operator for MDP$_n$ having components $\bigl\{ \sZ,\Lambda_n,\eta,c \bigr\}$. Analogous with $T$, it can be shown that $T_n$ is a contraction operator with modulus $\sigma = \alpha \beta$ mapping $C_w(\sZ)$ into itself. Let $J_n^{*} \in C_w(\sZ)$ (discounted value function of MDP$_n$) denote the fixed point of $T_n$.

The following theorem is the main result of this section which states that the discounted value function of MDP$_n$ converges to the discounted value function of the original MDP.

\begin{theorem}
\label{thm2}
For any compact set $K \subset \sZ$ we have
\begin{align}
\lim_{n\rightarrow\infty} \sup_{z\in K} |J_n^*(z) - J^*(z)| &= 0. \label{eq16} \\
\intertext{Therefore,}
\lim_{n\rightarrow\infty} |J_n^*(z) - J^*(z)| &= 0 \text{  } \text{ for all $z\in\sZ$}. \nonumber
\end{align}
\end{theorem}

To prove Theorem~\ref{thm2}, we need following results.

\begin{lemma}\label{nlemma0}
For any compact subset $K$ of $\sZ$ and for any $\varepsilon>0$, there exists a compact subset $K_{\varepsilon}$ of $\sZ$ such that
\begin{align}
\sup_{(z,a) \in K\times\sA} \int_{K_{\varepsilon}^c} w(y) \eta(dy|z,a) < \varepsilon. \label{tight}
\end{align}
\end{lemma}

\begin{proof}
Let us define the set of measures $\Xi$ on $\sZ$ as
\begin{align}
\Xi \coloneqq \biggl\{ Q(\,\cdot\,|z,a) : Q(D|z,a) = \int_{D} w(y) \eta(dy|z,a), \text{  }  (z,a) \in K\times\sA \biggr\}.\nonumber
\end{align}
Note that $\Xi$ is uniformly bounded since
\begin{align}
\sup_{(z,a) \in K\times\sA} \int_{\sZ} w(y) \eta(dy|z,a) \leq \alpha \sup_{z\in K} w(z) < \infty. \nonumber
\end{align}
If the mapping $Q: K\times\sA \ni (z,a) \mapsto Q(\,\cdot\,|z,a) \in \Xi$ is continuous with respect to the weak topology on $\Xi$, then $\Xi$ (being a continuous image of the compact set $K\times\sA$) is compact with respect to the weak topology. Then, by an extension of Prohorov's theorem to non-probability measures \cite[Theorem 8.6.2]{Bog07}, $\Xi$ is tight, completing the proof. Hence, we only need to prove the continuity of the mapping $Q$.

By Lemma~\ref{nlemma1}, for any $u \in C_w(\sZ)$, $\int_{\sZ} u(y) \eta(dy|z,a)$ is continuous in $(z,a)$. Let $(z_k,a_k) \rightarrow (z,a)$ in $K\times\sA$. Note that for any $g \in C_b(\sZ)$, $g w \in C_w(\sZ)$. Therefore, we have
\begin{align}
\lim_{k\rightarrow\infty} \int_{\sZ} g(y) Q(dy|z_k,a_k) &= \lim_{k\rightarrow\infty} \int_{\sZ} g(y) w(y) \eta(dy|z_k,a_k) \nonumber\\
&= \int_{\sZ} g(y) w(y) \eta(dy|z,a) = \int_{\sZ} g(y) Q(dy|z,a) \nonumber
\end{align}
proving that $Q(\,\cdot\,|z_k,a_k) \rightarrow Q(\,\cdot\,|z,a)$ weakly.
\end{proof}

\begin{lemma}\label{lemma3}
Let $\{u_n\}$ be a sequence in $C_w(\sZ)$ with $\sup_{n} \|u_n\|_w \coloneqq L <\infty$.
If $u_n$ converges to $u \in C_w(\sZ)$ uniformly on each compact subset of $\sZ$, then for any $f \in \rF$ and compact subset $K$ of $\sZ$ we have
\begin{align}
\lim_{n\rightarrow\infty} \sup_{z\in K} \biggl| \int_{\sZ} u_n(y) \eta(dy|z,f_n(z)) - \int_{\sZ} u(y) \eta(dy|z,f(z)) \biggr| = 0, \nonumber
\end{align}
where $f_n = \Upsilon_n(f)$ (see (\ref{neq6})).
\end{lemma}

\begin{proof}
Fix a compact subset $K$ of $\sZ$. Then for $K_{\varepsilon}$ as in Lemma~\ref{nlemma0},
\begin{align}
&\sup_{z\in K} \biggl| \int_{\sZ} u_n(y) \eta(dy|z,f_n(z)) - \int_{\sZ} u(y) \eta(dy|z,f(z)) \biggr| \nonumber \\
&\phantom{xxx}\leq \sup_{z\in K} \biggl| \int_{\sZ} u_n(y) \eta(dy|z,f_n(z)) -
\int_{\sZ} u(y) \eta(dy|z,f_n(z)) \biggr| \nonumber \\
&\phantom{xxxxxxxxxxxxxxxxx} + \sup_{z\in K} \biggl| \int_{\sZ} u(y) \eta(dy|z,f_n(z)) -
\int_{\sZ} u(y) \eta(dy|z,f(z)) \biggr| \nonumber \\
&\phantom{xxx}\leq \sup_{z\in K} \biggl| \int_{K_{\varepsilon}} u_n(y) \eta(dy|z,f_n(z)) - \int_{K_{\varepsilon}} u(y) \eta(dy|z,f_n(z)) \biggr| \nonumber \\
&\phantom{xxxxxxxxxxxxxxxxx}+ \sup_{z\in K} \biggl| \int_{K_{\varepsilon}^c} u_n(y) \eta(dy|z,f_n(z)) - \int_{K_{\varepsilon}^c} u(y) \eta(dy|z,f_n(z)) \biggr| \nonumber \\
&\phantom{xxxxxxxxxxxxxxxxxxxxxxxxxx}+ \sup_{z\in K} \biggl| \int_{\sZ} u(y) \eta(dy|z,f_n(z)) -
\int_{\sZ} u(y) \eta(dy|z,f(z)) \biggr|  \nonumber \\
&\phantom{xxx}\leq \sup_{y \in K_{\varepsilon}} |u_n(y) - u(y)| + 2 L \varepsilon + \sup_{z\in K} \biggl| \int_{\sZ} u(y) \eta(dy|z,f_n(z)) -
\int_{\sZ} u(y) \eta(dy|z,f(z)) \biggr|  \nonumber
\end{align}
Let us define $l(z,a) \coloneqq \int_{\sZ} u(y) \eta(dy|z,a)$. Since $u \in C_w(\sZ)$, by Lemma~\ref{nlemma1} $l$ is continuous, and therefore, uniformly continuous on $K\times\sA$. Note that in the last expression as $n\rightarrow\infty$: (i) the first term goes to zero since $u_n\rightarrow u$ uniformly on $K_{\varepsilon}$ and (ii) the last term goes to zero since $l$ is uniformly continuous on $K\times\sA$ and $f_n\rightarrow f$ uniformly. Then the result follows by observing that $\varepsilon$ is arbitrary.
\end{proof}

Let us define $v^0 = v_n^0 = 0$, and $v^{t+1} = T v^t$ and $v_n^{t+1} = T_n v_n^t$ for $t\geq1$; that is, $\{v^t\}_{t\geq1}$ and $\{v_n^t\}_{t\geq1}$ are successive approximations to the discounted value functions of the original MDP and MDP$_n$, respectively.
Lemma~\ref{nlemma2} implies that $v^t$ and $v^t_n$ are in $C_w(\sZ)$ for all $t$ and $n$.
By \cite[Theorem 8.3.6, p. 47]{HeLa99}, \cite[(8.3.34), p. 52]{HeLa99} and \cite[Section 8.5, p. 65]{HeLa99} we  have
\begin{align}
v^t(z) &\leq J^*(z) \leq M \frac{w(z)}{1-\sigma}, \label{eq3} \\
\|v^t - J^*\|_w &\leq M \frac{\sigma^t}{1-\sigma}, \label{eq4} \\
\intertext{and}
v_n^t(z) &\leq J_n^*(z) \leq M \frac{w(z)}{1-\sigma}, \label{eq5} \\
\|v_n^t - J_n^*\|_w &\leq M \frac{\sigma^t}{1-\sigma}. \label{eq6}
\end{align}
Since for each $n$ and $u$, $T u \leq T_n u$, we also have $v^t \leq v_n^t$  for all $t\geq1$ and $J^* \leq J_n^*$.

\begin{lemma}\label{lemma1}
For any compact set $K\subset\sZ$ and $t\geq1$, we have
\begin{align}
\lim_{n\rightarrow\infty} \sup_{z\in K} |v_n^t(z) - v^t(z)| = 0. \label{neq1}
\end{align}
\end{lemma}

\begin{proof}
We prove (\ref{neq1}) by induction. For $t=1$, the claim holds since $v^0 = v^0_n = 0$, and $c$ is uniformly continuous on $K\times\sA$ for any compact subset $K$ of $\sZ$. Assume the claim is true for $t\geq1$. We fix any compact set $K$. Let $f^*_t$ denote the selector of $T v^t = v^{t+1}$; that is,
\begin{align}
v^{t+1}(z) = T v^t(z) = c(z,f^*_t(z)) + \beta \int_{\sZ} v^t(y) \eta(dy|z,f^*_t(z)), \nonumber
\end{align}
and let $f^*_{t,n} \coloneqq \Upsilon_n(f^*_t)$ (see (\ref{neq6})). By (\ref{eq3}) and (\ref{eq5}) we have
\begin{align}
v^t(z) &\leq M \frac{w(z)}{1-\sigma} \label{neqa} \\
v^t_n(z) &\leq M \frac{w(z)}{1-\sigma}, \label{neq2}
\end{align}
for all $t$ and $n$. For each $n\geq 1$, we have
\begin{align}
&\sup_{z\in K} \bigl|v_n^{t+1}(z) - v^{t+1}(z)\bigr| \nonumber \\
&=\sup_{z\in K} \bigl(v_n^{t+1}(z) - v^{t+1}(z)\bigr) \text{  } \text{ (as $v^{t+1}\leq v_n^{t+1}$)} \nonumber \\
&= \sup_{z \in K} \biggl( \min_{\Lambda_n} \biggl[ c(z,a) + \beta \int_{\sZ} v_n^t(y) \eta(dy|z,a) \biggr] - \min_{\sA} \biggl[ c(z,a) + \beta \int_{\sZ} v^t(y) \eta(dy|z,a) \biggr] \biggr)  \nonumber \\
&\leq \sup_{z \in K} \biggl( \biggl[ c(z,f^*_{t,n}(z)) + \beta \int_{\sZ} v_n^t(y) \eta(dy|z,f^*_{t,n}(z)) \biggr] - \biggl[ c(z,f^*_t(z)) + \beta \int_{\sZ} v^t(y) \eta(dy|z,f^*_t(z)) \biggr] \biggr) \nonumber \\
&\leq \sup_{z\in K} \bigl|c(z,f^*_{t,n}(z)) - c(z,f^*_{t}(z))\bigr| + \beta \sup_{z\in K} \biggl| \int_{\sZ} v_n^t(y) \eta(dy|z,f^*_{t,n}(z)) -\int_{\sZ} v^t(y) \eta(dy|z,f^*_t(z)) \biggr|. \nonumber
\end{align}
Note that in the last expression, as $n\rightarrow\infty$ the first term goes to zero since $c$ is uniformly continuous on $K\times\sA$ and $f^*_{t,n}\rightarrow f^*_t$ uniformly, and the second term goes to zero by Lemma~\ref{lemma3}, (\ref{neqa}), and (\ref{neq2}).
\end{proof}

Now, using Lemma~\ref{lemma1} we prove Theorem~\ref{thm2}.

\begin{proof}[Proof of Theorem~\ref{thm2}]
Let us fix any compact set $K\subset\sZ$. Since $w$ is bounded on $K$, it is enough to prove $\lim_{n\rightarrow\infty} \sup_{z\in K} \frac{| J_n^*(z) - J^*(z) |}{w(z)} = 0$. We have
\begin{align}
&\sup_{z\in K} \frac{| J_n^*(z) - J^*(z) |}{w(z)} \leq \sup_{z\in K} \frac{| J_n^*(z) - v_n^t(z) |}{w(z)} + \sup_{z\in K} \frac{| v_n^t(z) - v^t(z) |}{w(z)} + \sup_{z\in K} \frac{| v^t(z) - J^*(z) |}{w(z)} \nonumber \\
&\leq 2 M \frac{\sigma^t}{1-\sigma} + \sup_{z\in K} \frac{| v_n^t(z) - v^t(z) |}{w(z)} \text{  } \text{ (by (\ref{eq4}) and (\ref{eq6}))} \nonumber.
\end{align}
Since $w\geq1$, $\sup_{z\in K} \frac{| v_n^t(z) - v^t(z) |}{w(z)} \rightarrow 0$ as $n\rightarrow\infty$ for all $t$ by Lemma~\ref{lemma1}. Hence, the last expression can be made arbitrarily small since $t\geq1$ is arbitrary and $\sigma \in (0,1)$, this completes the proof.
\end{proof}

\section{Near Optimality of Quantized Policies with Average Cost}
\label{sec2sub2}

In this section we consider the problem \textbf{(P)} for the average cost. We prove an approximation result analogous to Theorem~\ref{thm2}. To do this, some new assumptions are needed on the components of the original MDP in addition to the conditions in Assumption~\ref{as1}. A version of these assumptions were used in \cite{Veg03} and \cite{GoHe95} to study the existence of the solution to the Average Cost Optimality Equality (ACOE) and Inequality (ACOI). For any probability measure $\vartheta$ and measurable function $h$ on $\sZ$, let $\vartheta(h) \coloneqq \int_{\sZ} h(z) \vartheta(dz)$.

\begin{assumption}\label{as2}
Suppose Assumption~\ref{as1} holds with (\ref{eq2}) replaced by condition (e) below. Moreover, supposed there exist a probability measure $\lambda$ on $\sZ$ and a continuous function $\phi:\sZ\times\sA \rightarrow [0,\infty)$ such that
\begin{itemize}
\item [(e)] $\int_{\sZ} w(y) \eta(dy|z,a) \leq \alpha w(z) + \lambda(w) \phi(z,a)$ for all $(z,a) \in \sZ\times\sA$, where $\alpha \in (0,1)$.
\item [(f)] $\eta(D|z,a) \geq \lambda(D) \phi(z,a)$ for all $(z,a) \in \sZ\times\sA$ and $D \in \B(\sZ)$.
\item [(g)] The weight function $w$ is $\lambda$-integrable.
\item [(h)] $\int_{\sZ} \phi(z,f(z)) \lambda(dz) > 0$ for all $f \in \rF$.
 \end{itemize}
\end{assumption}

Any $f \in \rF$ gives rise to a time-homogenous Markov chain $\{z_t\}_{t=1}^{\infty}$ (state process) with the transition probability
$\eta(\,\cdot\,|z,f(z))$ on $\sZ$ given $\sZ$. For any $t\geq1$, let $\eta^t(\,\cdot\,|z,f(z))$ denote the $t$-step transition probability of this Markov chain given the initial point $z$. Hence, $\eta^t(\,\cdot\,|z,f(z))$ is recursively given by
\begin{align}
\eta^{t+1}(\,\cdot\,|z,f(z)) = \int_{\sZ} \eta(\,\cdot\,|y,f(y)) \eta^t(dy|z,f(z)). \nonumber
\end{align}

For any $z \in \sZ$, let
\begin{align}
V^*(z) \coloneqq \inf_{\varphi\in\Phi} V(\varphi,z). \nonumber
\end{align}
$V^*$ is called the average cost value function of the MDP. The following theorem is a consequence of \cite[Theorems 3.3 and 3.6]{Veg03}.

\begin{theorem}\label{thm3}
Under Assumption~\ref{as2} the following holds.
\begin{itemize}
\item [(i)] For each $f \in \rF$, the stochastic kernel $\eta(\,\cdot\,|z,f(z))$ is positive Harris recurrent with unique invariant probability measure $\nu_f$. Furthermore, $w$ is $\nu_f$-integrable, and therefore, $\rho_f  \coloneqq \int_{\sZ} c(z,f(z)) \nu_f(dz) <~\infty$.
\item [(ii)] There exist $f^{*} \in \rF$ and  $h^{*} \in C_w(\sZ)$ such that the triple $(h^{*},f^{*},\rho_{f^{*}})$ satisfies the average cost optimality equality (ACOE), i.e.,
\begin{align}
\rho_{f^{*}} + h^{*}(z) &= \min_{a\in\sA} \biggl[ c(z,a) + \int_{\sZ} h^{*}(y) \eta(dy|z,a) \biggr] \nonumber \\
&=  c(z,f^{*}(z)) + \int_{\sZ} h^{*}(y) \eta(dy|z,f^{*}(z)), \nonumber
\end{align}
and therefore,
\begin{align}
V^{*}(z) = \rho_{f^{*}}, \nonumber
\end{align}
for all $z\in\sZ$.
\end{itemize}
\end{theorem}

\begin{proof}
The only statement that does not directly follow from \cite[Theorems 3.3 and 3.6]{Veg03} is the fact: $h^{*} \in C_w(\sZ)$. Hence, we only prove this.

By \cite[Theorem 3.5]{Veg03}, $h^*$ is the unique fixed point of the following contraction operator with modulus $\alpha$
\begin{align}
F u(z) \coloneqq \min_{a \in \sA} \biggl[ c(z,a) + \int_{\sZ} u(y) \eta(dy|z,a) - \lambda(u) \phi(z,a) \biggr]. \nonumber
\end{align}
Since $\phi$ is continuous, by Lemma~\ref{nlemma1} the function inside the minimization is continuous in $(z,a)$ if $u \in C_w(\sZ)$. Then by Lemma~\ref{nlemma2}, $F$ maps $C_w(\sZ)$ into itself. Therefore, $h^* \in C_w(\sZ)$.
\end{proof}

This theorem implies that for each $f \in \rF$, the average cost is given by $V(f,z) = \int_{\sZ} c(y,f(y)) \nu_f(dy)$ for all $z\in\sZ$ (instead of $\nu_f$-a.e.).

\begin{remark}
If the state space $\sZ$ is compact and the transition probability $\eta(\,\cdot\,|x,a)$ has a strictly positive density $g(y|x,a)$ with respect to some probability measure $\vartheta$ which is continuous in $(y,x,a)$, then Assumption~\ref{as2} holds for
$w=1$, $\lambda = \vartheta$, and $\phi(z,a) = \min_{y \in \sZ} g(y|z,a)$.
\end{remark}

Note that all the statements in Theorem~\ref{thm3} are also valid for MDP$_n$ with an optimal policy $f_n^*$ and a canonical triplet $(h_n^*,f_n^*,\rho_{f_n^*})$. Analogous with $F$, define the contraction operator $F_n$ (with modulus $\alpha$) corresponding to MDP$_n$ as
\begin{align}
F_n u(z) \coloneqq \min_{a \in \Lambda_n} \biggl[ c(z,a) + \int_{\sZ} u(y) \eta(dy|z,a) - \lambda(u) \phi(z,a) \biggr], \nonumber
\end{align}
and therefore, $h_n^* \in C_w(\sZ)$ is its fixed point.

The next theorem is the main result of this section which states that the average cost value function, denoted as $V^*_n$, of MDP$_n$ converges to the average cost value function $V^*$ of the original MDP.

\begin{theorem}\label{mainthm2}
We have
\begin{align}
\lim_{n\rightarrow\infty} |V^*_n - V^*| = 0, \nonumber
\end{align}
where $V^*_n$ and $V^*$ are both constants.
\end{theorem}

Let us define $u^0 = u_n^0 = 0$, and $u^{t+1} = F u^t$ and $u_n^{t+1} = F_n u_n^t$ for $t\geq1$; that is, $\{u^t\}_{t\geq1}$ and $\{u_n^t\}_{t\geq1}$ are successive approximations to $h^*$ and $h_n^*$, respectively.
Lemma~\ref{nlemma2} implies that $u^t$ and $u^t_n$ are in $C_w(\sZ)$ for all $t$ and $n$.

\begin{lemma}\label{lemma_aux}
For all $u, v \in C_w(\sZ)$ and $n\geq1$, the following results hold: (i) if $u \leq v$, then $F u \leq F v$ and $F_n u \leq F_n v$; (ii) $F u \leq F_n u$.
\end{lemma}

\begin{proof}
Define the sub-stochastic kernel $\hat{\eta}$ by letting
\begin{align}
\hat{\eta}(\,\cdot\,|z,a) \coloneqq \eta(\,\cdot\,|z,a) - \lambda(\,\cdot\,) \phi(z,a). \nonumber
\end{align}
Using $\hat{\eta}$, $F$ and $F_n$ can be written as
\begin{align}
F u(z) &\coloneqq \min_{a \in \sA} \biggl[ c(z,a) + \int_{\sZ} u(y) \hat{\eta}(dy|z,a) \biggr], \nonumber \\
F_n u(z) &\coloneqq \min_{a \in \Lambda_n} \biggl[ c(z,a) + \int_{\sZ} u(y) \hat{\eta}(dy|z,a) \biggr]. \nonumber
\end{align}
Then the results follow from the fact that $\hat{\eta}(\,\cdot\,|z,a) \geq 0$  by Assumption~\ref{as2}-(f).
\end{proof}

Lemma~\ref{lemma_aux} implies that $u^0 \leq u^1 \leq u^2 \leq \ldots \leq h^*$ and $u_n^0 \leq u_n^1 \leq u_n^2 \leq \ldots \leq h_n^*$. Note that $\|u^1\|_w, \|u^1_n\|_w \leq M$ by Assumption~\ref{as1}-(d). Since
\begin{align}
\|h^*\|_w &\leq \|h^* - u^1\|_w + \|u^1\|_w = \|F h^* - F u^0\|_w + \|u^1\|_w \leq \alpha \|h^*\|_w + \|u^1\|_w \nonumber \\
\|h_n^*\|_w &\leq \|h_n^* - u_n^1\|_w + \|u_n^1\|_w = \|F_n h_n^* - F_n u_n^0\|_w + \|u_n^1\|_w \leq \alpha \|h_n^*\|_w + \|u_n^1\|_w, \nonumber
\end{align}
we have
\begin{align}
u^t(z) &\leq h^*(z) \leq M \frac{w(z)}{1-\alpha}, \nonumber \\
\intertext{and}
u_n^t(z) &\leq h_n^*(z) \leq M \frac{w(z)}{1-\alpha}. \nonumber
\end{align}
By inequalities above and the facts $\|u^t - h^*\|_w \leq \alpha^t \|h\|_w$ and $\|u_n^t - h_n^*\|_w \leq \alpha^t \|h_n\|_w$, we also have
\begin{align}
\|u^t - h^*\|_w &\leq M \frac{\alpha^t}{1-\alpha}, \nonumber \\
\intertext{and}
\|u_n^t - h_n^*\|_w &\leq M \frac{\alpha^t}{1-\alpha}. \nonumber
\end{align}
By Lemma~\ref{lemma_aux}, for each $n$ and $v \in C_w(\sZ)$, we have $F v \leq F_n v$. Therefore, by the monotonicity of $F$ and the fact $u^0 = u_n^0 =0$, we have
\begin{align}
u^t &\leq u_n^t \nonumber \\
h^* &\leq h_n^*, \label{aux50}
\end{align}
for all $t$ and $n$.

\begin{lemma}\label{lemma2}
For any compact set $K\subset\sZ$ and $t\geq1$, we have
\begin{align}
\lim_{n\rightarrow\infty} \sup_{z\in K} |u_n^t(z) - u^t(z)| = 0. \label{neq40}
\end{align}
\end{lemma}

\begin{proof}
Note that for each $t\geq1$, by the dominated convergence theorem and $\lambda(w) < \infty$, we have $\lambda(u_n^t) \rightarrow \lambda(u^t)$ if $u^t_n \rightarrow u^t$ pointwise. The proof can be finished using the same arguments as in the proof of Lemma~\ref{lemma1} and so we omit the details.
\end{proof}

\begin{lemma}\label{lemma4}
For any compact set $K\subset\sZ$, we have
\begin{align}
\lim_{n\rightarrow\infty} \sup_{z\in K} |h_n^*(z) - h^*(z)| = 0. \nonumber
\end{align}
\end{lemma}

\begin{proof}
The lemma can be proved using the same arguments as in the proof of Theorem~\ref{thm2}.
\end{proof}

Now, using Lemma~\ref{lemma4} we prove Theorem~\ref{mainthm2}

\begin{proof}[Proof of Theorem~\ref{mainthm2}]
Recall that $V^* = \rho_{f*}$ and $V^*_n = \rho_{f_n^*}$, and they satisfy the following ACOEs:
\begin{align}
h^*(z) + \rho_{f^*} &= \min_{a \in \sA} \biggl[c(z,a) + \int_{\sZ} h^*(y) \eta(dy|z,a) \biggr] = c(z,f^*(z)) + \int_{\sZ} h^*(y) \eta(dy|z,f^*(z)) \nonumber \\
h_n^*(z) + \rho_{f_n^*} &= \min_{a \in \Lambda_n} \biggl[c(z,a) + \int_{\sZ} h_n^*(y) \eta(dy|z,a) \biggr] = c(z,f_n^*(z)) + \int_{\sZ} h_n^*(y) \eta(dy|z,f_n^*(z)). \nonumber
\end{align}
Note that $h^*_n \geq h^*$ (see (\ref{aux50})) and $\rho_{f_n^*} \geq \rho_{f^*}$. For each $n$, let $f_n \coloneqq \Upsilon_n(f^*)$. Then for any $z \in \sZ$ we have
\begin{align}
\limsup_{n\rightarrow\infty} \bigl( h_n^*(z) + \rho_{f_n^*} \bigr) &= \limsup_{n\rightarrow\infty} \biggl( \min_{a \in \Lambda_n} \biggl[c(z,a) + \int_{\sZ} h_n^*(y) \eta(dy|z,a) \biggr] \biggr) \nonumber \\
&= \limsup_{n\rightarrow\infty} \biggl(c(z,f_n^*(z)) + \int_{\sZ} h_n^*(y) \eta(dy|z,f_n^*(z)) \biggr) \nonumber \\
&\leq \limsup_{n\rightarrow\infty} \biggl( c(z,f_n(z)) + \int_{\sZ} h_n^*(y) \eta(dy|z,f_n(z)) \biggr) \nonumber \\
&= c(z,f^*(z)) + \int_{\sZ} h^*(y) \eta(dy|z,f^*(z)) \label{auxxx3} \\
&= h^*(z) + \rho_{f^*} \nonumber \\
&\leq \liminf_{n\rightarrow\infty} \bigl( h_n^*(z) + \rho_{f_n^*} \bigr), \nonumber
\end{align}
where (\ref{auxxx3}) follows from Lemma~\ref{lemma3} and the fact that $h^*_n$ converges to $h^*$ uniformly on any compact subset $K$ of $\sZ$ and $\sup_{n} \|h_n^*\|_w \leq \frac{M}{1-\alpha}$. Since $\lim_{n\rightarrow \infty} h_n^*(z) = h^*(z)$ by Lemma~\ref{lemma4}, we have $\lim_{n\rightarrow \infty} \rho_{f_n^*} = \rho_{f^*}$. This completes the proof.
\end{proof}

\section{Application to Partially Observed MDPs}
\label{sec3}

In this section we apply the result obtained in Section~\ref{sec2} to partially observed Markov decision processes (POMDPs). Consider a discrete time POMDP with state space $\sX$, action space $\sA$, and observation space $\sY$, all Borel spaces. Let $p(\,\cdot\,|x,a)$ denote the transition probability of the next state given the current state-action pair is $(x,a)$, and let $r(\,\cdot\,|x)$ denote the transition probability of the current observation given the current state variable $x$. The one-stage cost function, denoted by $\tilde{c}$, is again a measurable function from $\sX \times \sA$ to $[0,\infty)$.

Define the history spaces
$\tilde{\sH}_{t}=(\sY\times\sA)^{t}\times\sY$, $t=0,1,2,\ldots$ endowed with their
product Borel $\sigma$-algebras generated by $\B(\sY)$ and $\B(\sA)$. A
\emph{policy} $\pi=\{\pi_{t}\}$ is a sequence of stochastic kernels
on $\sA$ given $\tilde{\sH}_{t}$. We denote by $\Pi$ the set of all policies.
Hence, for any initial distribution $\mu$ and policy $\pi$ we can think of POMDP as a stochastic process $\bigl\{ x_t,y_t,a_t \bigr\}_{t\geq0}$ defined on a probability space $\bigl( \Omega, \B(\Omega), P_{\mu}^{\pi} \bigr)$ where $\Omega = \tilde{\sH}_{\infty} \times \sX^{\infty}$, $x_t$ is a $\sX$-valued random variable, $y_t$ is a $\sY$-valued random variable, $a_t$ is a $\sA$-valued random variable, and $P_{\mu}^{\pi}$-almost surely they satisfy
\begin{align}
P_{\mu}^{\pi}(x_0\in\,\cdot\,)&=\mu(\,\cdot\,)\nonumber \\*
P_{\mu}^{\pi}(x_t\in\,\cdot\,|x_{[0,t-1]},y_{[0,t-1]},a_{[0,t-1]})&=P_{\mu}^{\pi}(x_t\in\,\cdot\,|x_{t-1},a_{t-1})=p(\,\cdot\,|x_{t-1},a_{t-1}) \nonumber \\
P_{\mu}^{\pi}(y_t\in\,\cdot\,|x_{[0,t]},y_{[0,t-1]},a_{[0,t-1]})&=P_{\mu}^{\pi}(y_t\in\,\cdot\,|x_{t})=r(\,\cdot\,|x_{t}) \nonumber \\
P_{\mu}^{\pi}(a_t\in\,\cdot\,|x_{[0,t]},y_{[0,t]},a_{[0,t-1]})&=\pi_t(\,\cdot\,|y_{[0,t]},a_{[0,t-1]}) \nonumber
\end{align}
where $x_{[0,t]}=(x_0,\ldots,x_t)$, $y_{[0,t]}=(y_0,\ldots,y_t)$, and $a_{[0,t]}=(a_0,\ldots,a_{t})$ ($t\geq1$).
Let $\tilde{J}(\pi,\mu)$ denote the discounted cost function of the policy $\pi \in \Pi$ with initial distribution $\mu$ of the POMDP.

It is known that any POMDP can be reduced to a (completely observable) MDP \cite{Yus76}, \cite{Rhe74}, whose states are the posterior state distributions or beliefs of the observer; that is, the state at time $t$ is
\begin{align}
\sPr\{x_{t} \in \,\cdot\, | y_0,\ldots,y_t, a_0, \ldots, a_{t-1}\} \in \P(\sX). \nonumber
\end{align}
We call this equivalent MDP the belief-MDP. The belief-MDP has state space $\sZ = \P(\sX)$ and action space $\sA$. The transition probability $\eta$ of the belief-MDP can be constructed as in Example~\ref{exm1} (see also \cite{Her89})
\begin{align}
\eta(\,\cdot\,|z,a) = \int_{\sY} 1_{\{F(z,a,y) \in \,\cdot\,\}} H(dy|z,a), \nonumber
\end{align}
where $F(z,a,y) \coloneqq \Pr\{x_{t+1} \in \,\cdot\, | z_t = z, a_t = a, y_{t+1} = y\}$, $H(\,\cdot\, | z,a) \coloneqq \Pr\{y_{t+1} \in \,\cdot\, | z_t = z, a_t = a\}$, and $z_t$ denotes the posterior distribution of the state $x_t$ given the past observations.
The one-stage cost function $c$ of the belief-MDP is given by
\begin{align}
c(z,a) \coloneqq \int_{\sX} \tilde{c}(x,a) z(dx). \label{eq8}
\end{align}
Hence, the belief-MDP is a Markov decision process with the components $(\sZ,\sA,\eta,c)$.

For the belief-MDP define the history spaces $\sH_{t}=(\sZ\times\sA)^{t}\times\sZ$, $t=0,1,2,\ldots$ as in Section~\ref{sec1}. Again, $\Phi$ denotes the set of all policies for the belief-MDP, where the policies are defined in an usual manner. Let $J(\varphi,\xi)$ denote the discounted cost function of policy $\varphi \in \Phi$ for initial distribution $\xi$ of the belief-MDP.

Notice that any history vector $h_t = (z_0,\ldots,z_t,a_0,\ldots,a_{t-1})$ of the belief-MDP is a function of the history vector $\tilde{h}_t = (y_0,\ldots,y_t,a_0,\ldots,a_{t-1})$ of the POMDP. Let us write this relation as
$i(\tilde{h}_t) = h_t$. 
Hence, for a policy $\varphi = \{\varphi_t\} \in \Phi$, we can define a policy $\pi^{\varphi} = \{\pi_t^{\varphi}\} \in \Pi$ as
\begin{align}
\pi_t^{\varphi}(\,\cdot\,|\tilde{h}_t) \coloneqq \varphi_t(\,\cdot\,|i(\tilde{h}_t)). \nonumber
\end{align}
Let us write this as a mapping from $\Phi$ to $\Pi$: $\Phi \ni \varphi \mapsto i(\varphi) = \pi^{\varphi} \in \Pi$. It is straightforward to show that the cost functions $J(\varphi,\xi)$ and $\tilde{J}(\pi^{\varphi},\mu)$ are the same. One can also prove that (see \cite{Yus76}, \cite{Rhe74})
\begin{align}
\inf_{\varphi \in \Phi} J(\varphi,\xi) &= \inf_{\pi \in \Pi} \tilde{J}(\pi,\mu) \label{eq7}
\end{align}
and furthermore, that if $\varphi$ is an optimal policy for belief-MDP, then $\pi^{\varphi}$ is optimal for the POMDP as well. Hence, the POMDP and the corresponding belief-MDP are equivalent in the sense of cost minimization. We will impose the following assumptions on the components of the original POMDP.

\begin{assumption}
\label{as3}
\item [(a)] The one stage cost function $\tilde{c}$ is continuous and bounded.
\item [(b)] The stochastic kernel $p(\,\cdot\,|x,a)$ is weakly continuous in $(x,a) \in \sX \times \sA$.
\item [(c)] The stochastic kernel $r(\,\cdot\,|x)$ is continuous in total variation, i.e., if $x_k \rightarrow x$, then $r(\,\cdot\,|x_k) \rightarrow r(\,\cdot\,|x)$ in total variation.
\item [(d)] $\sA$ is compact.
\end{assumption}

We refer the reader to \cite[Section 8]{FeKaZg14} for examples satisfying Assumption~\ref{as3}-(c). Note that by \cite[Proposition 7.30]{BeSh78}, the one stage cost function $c$, which is defined in (\ref{eq8}), is in $C_b(\sZ\times\sA)$ under Assumption \ref{as3}-(a)(b). Hence, the belief-MDP satisfies the conditions in Theorem~\ref{thm2} for $w=1$ if $\eta$ is weakly continuous. The following theorem is a consequence of \cite[Theorem 3.7, Example 4.1]{FeKaZg14} and Example~\ref{exm1}.

\begin{theorem}\label{thm6}
\begin{itemize}
\item [ ]
\item [(i)] Under Assumption \ref{as3}-(b)(c), the stochastic kernel $\eta$ for belief-MDP is weakly continuous in $(z,a)$.
\item [(ii)] If we relax the continuity of the observation channel in total variation to setwise or weak continuity, then $\eta$ may not be weakly continuous even if the transition probability $p$ of POMDP is continuous in total variation.
\item [(iii)] Finally, $\eta$ may not be setwise continuous in $a$, even if the observation channel is continuous in total variation.
\end{itemize}
\end{theorem}

Part(i) of Theorem~\ref{thm6} implies that belief-MDP satisfies conditions in Theorem~\ref{thm2}. However, note that continuity of the observation channel in total variation in Assumption~\ref{as3} cannot be relaxed to weak or setwise continuity. On the other hand, the continuity of the observation channel in total variation is not enough for the setwise continuity of $\eta$. Hence, results in \cite{SaLiYu13-2} cannot be applied to the POMDP we consider even though we put a fairly strong condition on the observation channel.

\begin{theorem}
\label{thm5}
Suppose Assumption~\ref{as3} holds for the POMDP. Then we have
\begin{align}
\lim_{n\rightarrow\infty} |J^*_n(z) - J^*(z)| = 0 \text{  } \text{ for all $z\in\sZ$}, \nonumber
\end{align}
where $J^*_n$ is the discounted value function of the belief-MDP with the components $\bigl\{ \sZ, \Lambda_n,\eta,c \bigr\}$ and
$J^*$ is the discounted value function of the belief-MDP with the components $\bigl\{ \sZ, \sA,\eta,c \bigr\}$.
\end{theorem}

The significance of Theorem~\ref{thm5} is reinforced by the following observation. If we define $D\Pi\Q(\Lambda_n)$ as the set of deterministic policies in $\Pi$ taking values in $\Lambda_n$, then the above theorem implies that for any given $\varepsilon>0$ there exists $n\geq1$ and $\pi^{*} \in D\Pi\Q(\Lambda_n)$ such that
\begin{align}
\tilde{J}(\pi^{*},\mu) < \min_{\pi \in \Pi} \tilde{J}(\pi,\mu) + \varepsilon, \nonumber
\end{align}
where $\pi^{*} = \pi^{\varphi^{*}}$. This means that even when is an information transmission constraint from the controller to the plant, one might get $\varepsilon$-close to the value function for any small $\varepsilon$ by quantizing the controller's actions and sending the encoded levels.

\section{Discussion}
\label{sec4}

In this paper, we considered the finite-action approximation of stationary policies for a discrete-time Markov decision process with either discounted or average costs. Under mild weak continuity assumptions it was shown that if one uses a sufficiently large number
of points to discretize the action space, then the resulting finite-action MDP can approximate the original model with arbitrary precision.
The results obtained for the discounted cost were also applied to the finite-action approximation problem for POMDPs.

One direction for future work is to further investigate the problem \textbf{(P)} for the average cost under specific conditions for POMDPs, so that the results obtained for the average cost can be applicable to the belief-MDPs. In this case, a possible solution methodology is to investigate conditions on the POMDP under which the Markov chain arising from the belief-MDP with a stationary policy is ergodic and hence has a unique invariant measure.

\bibliographystyle{elsarticle-num}
\bibliography{references}

\end{document}